\documentclass[]{article}

\usepackage{layout}
\usepackage{amsthm,amsmath,graphicx,latexsym,amssymb,amscd,amsfonts}

\numberwithin{equation}{section}

\theoremstyle{plain}
\newtheorem{theorem}{Theorem}[section]

\newtheorem{lemma}[theorem]{Lemma}

\newtheorem{remark}[theorem]{Remark}

\begin{document}

\title{\textbf{\Large Fractional nonlinear filtering problems and their associated fractional Zakai equations}}
\author{ { {\textsc{Sabir Umarov}}}\thanks{
Department of Mathematics,
                        University of New Haven, 300 Boston Post Rd, West Haven, CT 06511; 
                        sumarov@newhaven.edu},
~  \textsc{Frederick Daum$^{\dagger}$, Kenric Nelson} \thanks{Raytheon, 235 Presidential Way
Woburn, MA 01801; Frederick$_{-}$E$_{-}$Daum@raytheon.com, Kenric$_{-}$P$_{-}$Nelson@raytheon.com} 
}
\date{ }
\maketitle
\vspace{-6mm}
\renewcommand{\thefootnote}{\fnsymbol{footnote}}

\begin{abstract}
In this paper we discuss fractional generalizations of the filtering problem. The "fractional" nature comes from time-changed state or observation processes, basic ingredients of the filtering problem. The mathematical feature of the fractional filtering problem emerges as the Riemann-Liouville fractional derivative in the associated Zakai equation. We discuss fractional generalizations of the filtering problem whose state and observation processes are driven by time-changed Brownian motion or/and  L\'evy process.


\end{abstract}

\section{Introduction}

The filtering problem is formulated as follows. Let $(\Omega, \mathcal{F}, P)$ be a probability space with a state space $\Omega,$ sigma-algebra $\mathcal{F},$ and probability measure $P.$ Let $X_t:\Omega \to R^n$ be an $R^n$-valued stochastic process defined on $(\Omega, \mathcal{F}, P)$ and called \textit{a state process.} We assume that $X_t$ is governed by the stochastic differential equation
\begin{equation} \label{state}
dX_t=b(t,X_t)dt+g(t,X_t)dB_t,
\end{equation}
with initial condition $X_{t=0}=X_0,$ where $X_0$ is a random variable independent of Brownian motion $B_t,$ the functions $b(t,x)$ and $g(t,x)$ defined for $t>0$ and $x \in R^n$ satisfy some growth and continuity conditions.  
 The state process in the filtering problem can not be observed directly. Suppose $Z_s, \, s \le t,$ is $R^m$-valued stochastic process called \textit{observation process} and related to the process $X_t$ in the noisy environment. The observation process $Z_t$ can be expressed through a stochastic differential equation of the form
\begin{equation}
\label{obsproc}
dZ_t=h(t, X_t)dt +dW_t, ~~ Z_0=0,
\end{equation}
where $h(t,x), \, t > 0, \, x \in R^n,$ is a function satisfying appropriate growth and continuity conditions, and $W_t$ is an $m$-dimensional Brownian motion independent of $B_t.$
Let ${\cal{Z}}_t$ be a $\sigma$-algebra generated by the observation process $Z_t.$ We assume that $X_t$ is  
${\cal{Z}}_t$-measurable. 
The filtering problem is to find the best estimation of $X_t$ at time $t$ given $Z_t,$ in the mean square sense. Namely, to find a stochastic process $X^{\ast}_t$ such that
\[
E[\|X_t-X^{\ast}_t\|^2] = \inf {E[\|X_t-Y_t\|^2]},
\] 
where $E$ is the expectation with respect to the probability measure $P$ and $\inf$ is taken over all ${\cal{Z}}_t$-measurable stochastic processes $Y_t \in L_2(P).$ It follows from the abstract theory of functional analysis that $X^{\ast}_t$ is the projection of $X_t$ onto the space of stochastic processes $\mathcal{L}(Z_t)=\{Y\in L_2(P): Y_t ~ \mbox{is} ~ {\cal{Z}}_t-\mbox{measurable}\}.$ The latter can be written in the form $X^{\ast}_t=E[X_t|{\mathcal{Z}}_t].$

Filtering problems arise in many engineering models. One simple example is transmitting of modulated signals. These signals are received with an effect of noisy environment. The received signal has to be filtered in order to be realizable. Thus in this situation filtering problem is about the best estimation of the stochastic process (the modulated signal transmitted in the noisy environment)  given additional information obtained via measurement of parameters of the process (of the signal). 


The filtering problem is called \textit{linear} if the functions $b(t,x)$ and $g(t,x)$ depend on $x$ linearly. The linear filtering problem was studied by Kalman and Bucy \cite{KalmanBucy} in the 1960th. They reduced the linear filtering problem to a linear SDE and a deterministic Riccati type differential equation.
In the case of \textit{non linear filtering} Kushner \cite{Kushner}, Lipster and Shiryaev \cite{LipsterShiryaev}, and Fujisaki, Kallianpur and Kunita \cite{FKK}
 (see also \cite{Rozovskii}) obtained a non linear infinite dimensional stochastic differential equations for the posterior conditional density of $X_t$ given ${\mathcal{Z}}_t.$ However, 
(1) it is not easy to solve these equations, and 
(2) it is computationally 'expensive' due to the two stage calculation procedure (prediction and correction) in the real time.  
Later Zakai \cite{Zakai} obtained a simpler form of the stochastic differential equation for the posterior unnormalized conditional density 
$\Phi(t,x)=p(t,x|{\mathcal{Z}}_t)$ for $X_t$ in the following form:
\begin{equation} \label{Zakai10}
\Phi(t,x)=P(X_0=x)+\int_0^t A^{\ast}\Phi(s,x)ds+\sum_{k=1}^m  \int_0^t h_k(x)\Phi(s,x) dZ^{(k)}_s,
\end{equation}
where the operator $A^{\ast}$ is the dual of the infinitesimal generator $A$ (see Section \ref{Sec: Generalization}) of the Markov process $X_t,$ and $h_k(x), k=1,...,m,$ are components of the random vector-function in equation \eqref{observ}.  Equation \eqref{Zakai10} can be written in the differential form as follows
\begin{equation} \label{Zakai15}
d\Phi(t,x)= A^{\ast} \Phi(t,x) dt+ \sum_{k=1}^m  h_k(x)\Phi(t,x) dZ^{(k)}_t, ~~ \Phi(0,x)=P(X_0=x).
\end{equation}
Equation \eqref{Zakai10} (or \eqref{Zakai15}) is a linear stochastic partial differential equation, and therefore, the methods of solution of linear equations are applicable, including some explicit forms for the solution.

Daum \cite{Daum1987,Daum2010,Daum2012} developed algorithms for solution of nonlinear filtering problems in the class of distributions from the Gaussian family. In paper \cite{Daum1987} he reduced
a solution of the Zakai equation to a solution of the Fokker-Planck equation and a
deterministic matrix Riccati equation, generalizing the classical result of Kalman
and Bucy. In recent works \cite{Daum2010,Daum2012} particle flow algorithms were suggested which
provide several orders of magnitude improvement in the processing of particle filters
by computing Bayes' rule as a flow of the logarithm of the conditional density from the prior to the
posterior, and he derived the corresponding flow of particles from the prior to the posterior.

All the works mentioned above relate to filtering problems with the underlying Gaussian processes. However, many processes naturally arising in the modern science (in particular, in biology, genetics, finance) and engineering  do not obey the Gaussian law. Filtering problems with the state and/or observation processes driven by a L\'evy process were discussed in recent publications \cite{Handbook,Popa,Proske}.

In this paper we discuss fractional generalizations of the filtering problem to the case when the state and observation processes are driven by time-changed Brownian motion or L\'evy processes. Fractional model of the filtering problem significantly extends the scope of the filtering problems both theoretically and their engineering and other applications.
As a time change process we consider the inverse of the L\'evy stable subordinator with the stability index $\beta$ or their mixtures. We note that these processes are not Gaussian (exponential) and are not L\'evy processes. These processes are special type of semimartingales. The associated Zakai equations in this case involve a fractional order derivative  in the sense of Riemann-Liouville. For instance, if the driving process is a time-changed Brownian motion where the time-change process $T_t$ is the first hitting time of the stable L\'evy subordinator with the stability index $\beta, ~ 0<\beta<1,$ then the associated Zakai equation is given by the following partial stochastic differential equation driven by a semimartingale $Z_{T_t}$ (see for details in Section \ref{Sec: FrZakai}):
\begin{equation} \label{fracSDE}
d\Phi(t,x)= {\mathcal{D}}_t^{1-\beta} A^{\ast} \Phi(t,x) dt+ \sum_{k=1}^m  h_k(x)\Phi(t,x) dZ^{(k)}_{T_t}, ~~ \Phi(0,x)=P(X_0=x),
\end{equation}
where $A^{\ast}$ is the adjoint operator of the infinitesimal generator of the Markovian process $X_t,$ and ${\mathcal{D}}_t^{1-\beta}$ is the fractional differentiation operator of order $1-\beta$ in the sense of Riemann-Liouville. 
By definition, the fractional derivative of order $\beta, ~ 0<\beta<1,$ in the Riemann-Liouville sense is ${\mathcal{D}}_t^{1-\beta} = \frac{d}{dt} J^{\beta},$ where the fractional integration operator $J^{\beta}$ is defined as
\[
J^{\beta} f(t) =\frac{1}{\Gamma(\beta)} \int_0^t (t-\tau)^{\beta-1}f(\tau) d \tau, \, t>0.
\]
 Due to the presence of the fractional derivative ${\mathcal{D}}_t^{1-\beta}$ in \eqref{fracSDE} it is natural to call such a stochastic differential equation a \textit{fractional  Zakai type equation.} Obviously, the usual Zakai equation  is recovered if $\beta \to 1.$ 

{In the mathematical literature stochastic differential equations driven by a fractional Brownian motion are also called {\it fractional.} However, the nature of stochastic differential equations \eqref{fracSDE}, that is fractional Zakai type stochastic differential equations, totally different from the nature of those fractional stochastic differential equations driven by a fractional Brownian motion.}


\section{Generalized filtering problems}
\label{Sec: Generalization}

\textit{2.1. Generalization of the filtering problem with a L\'evy processes.}
Many processes in the modern science and engineering do not obey the Gaussian law  for the state process in equation \eqref{state} and for the observation process in equation \eqref{observ}. Let $L_t, ~ t\ge 0,$ be an $n$-dimensional L\'evy process.
L\'evy
processes can be characterized by the L\'evy-It\^o decomposition theorem, which
states that
\begin{equation} \label{levy-ito}
L_t=b_0 t + \sigma B_t +\int_{|w|<1} w \tilde{N}(t,dw) +
\int_{|w|\ge1} w N(t,dw),
\end{equation}
where $b_0 \in {\mathbb{R}^n},$ $\sigma$ is an $n \times m$-matrix such that
$\sigma \sigma^{T}=\Sigma$, $B_t$ is an $m$-dimensional Brownian
motion, and $N_t$ and $\tilde{N}_t$ are a compound Poisson random measure and
a compensated Poisson martingale-valued measure, respectively \cite{Sato}.
It is well known \cite{Sato} that L\'evy processes have a c\`adl\`ag
modification and are semimartingales. 
Any L\'evy process is uniquely defined by a triple  $(b,\Sigma, \nu),$ where  $b \in \mathbb{R}^n,$ $\Sigma$ is a
nonnegative definite matrix,  and a measure $\nu$ defined on
$\mathbb{R}^n \setminus \{0\}$ such that $\int \min (1,|x|^2) d \nu
<\infty$. L\'evy processes can also be
characterized by the L\'evy-Khintchine formula
in terms of its characteristic function $\Phi_t(\xi)= E(e^{i \xi L_t}) = e^{t \Psi
(\xi)}$, with
\begin{equation}
\label{levy-khintchin} \Psi(\xi) = i(b,\xi) - \frac{1}{2} (\Sigma
\xi, \xi) + \int_{\mathbb{R}^n \setminus \{0\}} \hspace{-1mm}(e^{i(w, \xi)}- 1-i
(w,\xi) \chi_{(|w| \le 1)}(w))\nu(dw).
\end{equation}
The function $\Psi$ is called the \textit{L\'evy symbol} of
$L_t.$  For any L\'evy process, its
L\'evy symbol is continuous, hermitian, conditionally positive
definite
 and $\Psi(0)=0.$  The infinitesimal generator of the L\'evy process
with characteristics $(b,\Sigma, \nu)$ is a pseudo-differential
operator
${A}={A}(\mathbf{D}_x)$ with the symbol
$\Psi(\xi)$ defined in (\ref{levy-khintchin}).

A natural generalization of the filtering problem \eqref{state}-\eqref{obsproc} is to replace Brownian motions $B_t$ in \eqref{state} and $W_t$ in \eqref{obsproc} with L\'evy processes $L_t$ and $M_t,$ respectively. Namely, consider a state process $X_t$ governed by the L\'evy process $L_t:$
\begin{align} 
     X_t &=X_0 + \int_0^t b(X_{s-}) ds + \int_0^t \sigma(X_{s-}) dB_s \notag \\
          &+ \int_0^t \int_{|w|<1} H(X_{s-},w)\tilde{N}(ds, dw)+ \int_0^t \int_{|w|\geq 1} K(X_{s-},w) N(ds, dw),  \label{stateLevy}
\end{align}
where $X_0$ is a random variable independent of $B_t$ and
$N(t,\cdot);$ the continuous mappings $b:\mathbb{R}^n \to
\mathbb{R}^n,$ $\sigma: \mathbb{R}^n\to \mathbb{R}^{n \times m},$
$H:\mathbb{R}^n \times \mathbb{R}^n \to \mathbb{R}^n$ and
$K:\mathbb{R}^n\times \mathbb{R}^n \to \mathbb{R}^n$ satisfy the Lipschitz and linear growth
conditions. 
The infinitesimal generator $\mathcal{A}$ of the process $X_t$ is a pseudo-differential operator with the symbol
\begin{align}
\Psi(x,\xi)
     &= i(b(x),\xi) - \frac{1}{2} (\Sigma (x) \xi, \xi) \notag \\
     & + \int_{\mathbb{R}^n \setminus \{0\}} \hspace{-2mm} (e^{i(G(x,w),\xi)}- 1-i(G(x,w), \xi) \chi_{(|w| < 1)}(w))\nu(dw), \label{levysymbol}
\end{align}
where $G(x,w)=H(x,w)$ if $|w|<1,$ and $G(x,w)=K(x,w)$ if $|w|\ge 1$
\cite{Applebaum}. By definition, a pseudo-differential operator $\mathcal{A}$ with the symbol $\Psi(x,\xi)$ is
\begin{equation} \label{Levygenerator}
\mathcal{A} \varphi(x)=\frac{1}{(2\pi)^2} \int_{R^n} \Psi(x,\xi) \hat{\varphi}(\xi)e^{-ix \xi} d\xi,
\end{equation}
where $\hat{\varphi}$ is the Fourier transform of $\phi$ in the domain of $\mathcal{A}$ (see details in \cite{Applebaum,Hermander}).
Let the observation process is given by 
\begin{align} 
     Z_t &= \int_0^t \mu(s, X_{s-}) ds + \int_0^t \nu(s, X_{s-}) dW_s \notag \\
          &+ \int_0^t \int_{|w|<1} g(X_{s-},w)\tilde{M}(ds, dw)+ \int_0^t \int_{|w|\geq 1} f(X_{s-},w) M(ds, dw), \label{observLevy}
\end{align}
where Brownian motion $W_t$ is independent of $B_t$ in equation \eqref{stateLevy}, the measures $M_t$ and $\tilde{M}_t$ are a compound Poisson random measure and
a compensated Poisson martingale-valued measure,  and mappings $\mu(t,x), ~ \nu(t,x), ~ g(t,x,w),$ and $f(t,x,w)$ satisfy the Lipschitz and linear growth conditions. 
Let ${\mathcal{Z}}_t$ be the sigma-algebra generated by the process $Z_s, \, 0 \le s \le t.$ 
Now the generalized filtering problem is formulated as follows: 
\textit{find the best estimation of $X_t$ given $\mathcal{Z}_t.$}

\medskip

Particular cases of this problem is discussed in Chapter 4 of \cite{Handbook} and in papers \cite{Popa,Proske}. Namely, consider the following two cases:
\begin{enumerate}
\item[(1)] the state process is driven by a L\'evy process, and
\item[(2)] the observation process is driven by a L\'evy process.
\end{enumerate}
In the first case the filtering model is formulated as follows: The state process is given by stochastic differential equation \eqref{stateLevy} driven by a L\'evy process,
and the observation process is given by 
\begin{equation} \label{obser_1case}
Z_t=Z_0+W_t+\int_0^t h(X_s)ds,
\end{equation}
where $X_t$ and $Z_t$ are respectively an $R^n$-valued and $R^m$-valued stochastic processes. Then (see \cite{Popa}) the unnormalized conditional distribution of $f(X_t)$ given $\mathcal{Z}_t$, that is ${\phi}_t (f)=E[f(X_t)|\mathcal{Z}_t]$ under some conditions, satisfies the Zakai type stochastic partial differential equation
\begin{equation} \label{Zakai11}
{\phi}_(f)=\phi_0(f)+\int_0^t {\phi}_s(Af)ds + \sum_{k=1}^m\int_0^t {\phi}_s({fh_k})dZ_{ks},
\end{equation}
where $\mathcal{A}$ is the infinitesimal generator of the process $X_t$ defined in \eqref{Levygenerator},  and $h_k$ and $Z_k$ are components of vectors $h(x)=(h_1(x),...,h_m(x))$ and $Z_t=(Z_{1t},...,Z_{mt}).$

In the second case of the filtering model the state process is given by
\begin{equation}\label{state_2case}
X_t=X_0+\int_0^t b(X_s) ds +\sigma(X_s)dB_s,
\end{equation}
and the observation process is
\begin{align} \label{observ_2case}
     Z_t =Z_0 &+ \int_0^t h(s,X_{s}) ds + W_t + \int_{R} w N_{\lambda}(dt,dw),
\end{align}
where $N_{\lambda}$ is an integer valued random measure with predictable compensator $\lambda(t,X_t,\omega)dt d\nu,$ with a L\'evy measure $\nu.$ Let $\Phi(t,x)$ be a filtering density, that is for arbitrary infinitely differentiable function $f$ with a compact support the relation
\[
\phi_t(f)=\int_{R^n} f(x) \Phi(t,x)dx
\]
holds. Then (see \cite{Proske}) the corresponding Zakai type equation has the form
\begin{align} \label{Zakai2}
{\Phi}(t,x) &= p_0(x)+\int_0^t A^{\ast} {\Phi}(s,x)ds \notag\\
               &+ \int_0^t h(s,x)\Phi(s,x)dB_s + \int_0^t \int_R (\lambda(s,x,w)-1) {\Phi}(s,x)\tilde{N}(ds,dw),
\end{align}
where $A^{\ast}$ is the dual of the infinitesimal generator $A$ of $X_t$ and 
$$\tilde{N}(ds,dw)=N(ds,dw)-dsd\nu.$$ We note that in this case due to absence of jump components of the state process $X_t$ its infinitesimal generator $A$ is not a pseudo-differential operator.  The operator $A$ is 
a second order elliptic differential operator
\begin{equation} \label{operator}
A\varphi(x)=\frac{1}{2}\sum_{i,k=1}^n a_{i k}(x) \frac{\partial^2 \varphi(x)}{\partial x_i \partial x_k} +\sum_{k=1}^n b_k(x) \frac{\partial \varphi(x)}{\partial x_k},
\end{equation}
with the coefficients $a_{i k}(x),$  entries of the matrix-function $a(x)=\{a_{i k}(x)\}_{i,k=1}^n$ obtained by multiplying $\sigma(x)$ by its transpose  $\sigma(x)^{t}.$

It is not hard to verify that both cases recover the classical Zakai equation \eqref{Zakai10} if the jump component of the L\'evy process is absent, that is if $\nu \equiv 0.$

\bigskip

\textit{2.2. Generalization of the filtering problem with time-changed L\'evy processes.}
We note that the driving stochastic processes in the filtering model \eqref{stateLevy} and \eqref{observLevy} are semimartingales with independent increments. If $T$ is a L\'evy subordinator then $L_T$ is still a L\'evy process. Therefore, replacement of L\'evy processes $L_t$ and $M_t$ in \eqref{stateLevy} and \eqref{observLevy} with their time-changed ones $L_{T_1}$ and $M_{T_2},$ where $T_1$ and $T_2$ are L\'evy subordinators, does not expand the scope of the filtering models given by equations \eqref{stateLevy} and \eqref{observLevy}. If $T$ is an inverse to a L\'evy subordinator then the time-changed process $L_T$ is no longer a L\'evy process.  However, it is a semimartingale. Therefore, stochastic integrals driven by time-changed processes $L_T,$ where $T$ is the inverse  to a stable L\'evy subordinator, are well defined. 

Consider the following model of the filtering problem driving process of which is a time-changed L\'evy process. The state process in this context is given by
\begin{align} 
     X_t &=X_0 + \int_0^t b(X_{s-}) dT_s + \int_0^t \sigma(X_{s-}) dB_{T_s} \notag \\
          &+ \int_0^t \int_{|w|<1} H(X_{s-},w)\tilde{N}(dT_s, dw)+ \int_0^t \int_{|w|\geq 1} K(X_{s-},w) N(dT_s, dw),    \label{stateLevyFrac}
\end{align}
and the observation process is
\begin{align} 
     Z_t & = \int_0^t \mu(s, X_{s-}) dE_s + \int_0^t \nu(s, X_{s-}) dW_{E_s} \notag \\
          &+ \int_0^t \int_{|w|<1} g(X_{s-},w)\tilde{M}(dE_s, dw)+ \int_0^t \int_{|w|\geq 1} f(X_{s-},w) M(dE_s, dw),\label{observLevyFrac}
\end{align}
where $T_t$ and $E_t$ are inverse processes to stable L\'evy subordinators. Notice, that if $T_t=t$ and $E_t=t$ then the filtering model \eqref{stateLevyFrac} and \eqref{observLevyFrac} represents the model \eqref{stateLevy} and \eqref{observLevy} for filtering problem driven by L\'evy processes. Therefore, replacement of L\'evy processes $L_t$ and $M_t$ in the filtering model \eqref{stateLevy}, \eqref{observLevy} with $L_{T}$ and $M_E,$ respectively, where $T$ and $E$ are inverse L\'evy subordinators does expand essentially the scope of the model \eqref{stateLevy}, \eqref{observLevy}.

In Section \ref{Sec: FrZakai} we will solve this problem under certain constraints. We will need some preliminary facts on L\'evy subordinators and their inverses. 
Let
$E_t$ be the first hitting time  process for a stable subordinator
$D_t$ with stability index ${\beta \in(0,1)}.$ The process $E_t$ is
also called an inverse to $D_t.$ The relation between $E_t$ and
$D_t$ can be expressed as $E_t= \min \{\tau: D_{\tau} \ge t\}.$ The
process $D_{t}, ~ t \ge 0,$ is a self-similar L\'evy process with
$D_0=0,$ that is $D_{ct}=c^{\frac 1 \beta} D_t$ as processes in the
sense of finite dimensional distributions, and its Laplace transform
is $\mathbb{E}(e^{-sD_t})=e^{-ts^{\beta}}.$ The density
$f_{D_{_1}}(\tau)$ of $D_1$ is infinitely differentiable on $(0,
\infty),$ with the following asymptotics at zero and infinity
\cite{MinardiLuchkoPagnini,UzhaykinZolotarev}:
\begin{align}
&f_{{D_{_1}}}(\tau) \sim \frac{({\frac \beta
\tau})^{\frac{2-\beta}{2(1-\beta)}}}{\sqrt{2\pi \beta (1-\beta)}} \,
e^{-(1-\beta)({\frac \tau  \beta})^{-\frac{\beta}{1-\beta}}}, \,
\tau
\to 0; \label{atzero}\\
&f_{D_{_1}}(\tau) \sim \frac{\beta}{\Gamma(1-\beta) \tau^{1+\beta}},
\, \tau \to \infty. \label{atinfinity}
\end{align} Since
$D_t$ is strictly increasing, its inverse process $E_t$ is
continuous and nondecreasing, but not a L\'evy process. Likewise for any L\'evy process $L_t$ the
time-changed process $L_{E_t}$ is also not a L\'evy process (see
details in \cite{HKU,HU}). 

Let  $g_t(\tau)$ be the density function of $E_t$ for each fixed
$t>0.$ If $f_{D_{_1}}(t)$ is the density function of $D_1,$ then
\begin{equation} \label{relation2}
g_t(\tau)=-\frac{\partial}{\partial \tau}
Jf_{D_{_1}}(\frac{t}{\tau^{1 / \beta}}) = -\frac{\partial}{\partial
\tau} \int_0^{\frac{t}{\tau^{1 / \beta}}} f_{D_{_1}}(u)du
=\frac{t}{\beta \tau^{1+{\frac 1 \beta}}}f_{D_{_1}}(\frac{t}{\tau^{\frac 1
 \beta}}).
 \end{equation}
Since $f_{D_{_1}}\hspace{-0.5mm}(u) \in C^{\infty}(0,\infty),$ it follows from
representation \eqref{relation2} that $g_t(\tau) = \varphi(t,\tau) \in
C^{\infty}(\mathbb{R}_+^2),$ where $\mathbb{R}_+^2=(0,\infty)\times
(0,\infty).$ Further properties of $g_t(\tau)$ are represented in
the following lemma.

\begin{lemma} \label{lemproperties}
Let $g_t(\tau)$ be the function given in \eqref{relation2}.  Then
\begin{enumerate}
\item[$(a)$] $\lim_{t\to +0}g_t(\tau)=\delta_0(\tau)$
~ in the sense of the topology of the space of tempered
distributions $\mathcal{D}^{\prime}(\mathbb{R});$ \vspace{0.8mm}
\item[$(b)$] $\lim_{\tau \to +0}g_t(\tau)=\frac{t^{-\beta}}{\Gamma(1-\beta)}, ~ t>0;$\vspace{0.8mm}
\item[$(c)$] $\lim_{\tau \to \infty}g_t(\tau) = 0, ~ t > 0;$\vspace{0.8mm}
\item[$(d)$] $\mathcal{L}_{t \to s}[g_t(\tau)](s)=s^{\beta-1}e^{-\tau s^{\beta}}, ~ s>0, ~ \tau \ge 0,$
\end{enumerate} 
where $\mathcal{L}_{t \to s}$ denotes the Laplace transform with
respect to the variable $t$.
\end{lemma}

\begin{lemma} \label{lemmain}
Function $g_t(\tau)$ defined in \eqref{relation2} for each $t>0$
satisfies the following equation
\begin{equation}
\label{relation399} D_{\ast,t}^{\beta}g_{t}(\tau)=
-\frac{\partial}{\partial \tau} g_t(\tau)
-\frac{t^{-\beta}}{\Gamma(1-\beta)}\delta_0 (\tau), 
\end{equation}
in the sense of tempered distributions.
\end{lemma}

We refer the reader to papers \cite{HKU,HU} for proofs of these two lemmas.

\section{Fractional filtering problem. Main results}
\label{Sec: FrZakai}

First, for simplicity we consider the  filtering problem with the state process given in the differential form as
\begin{equation} \label{Conjecture1_state}
dX_t=b(X_t)dT_t+\sigma(X_t)dB_{T_t}, ~ X_{t=0}=X_0,
\end{equation}
and driven by a time-changed Brownian motion with drift, where $T_t$ is the inverse of the L\'evy stable subordinator 
with the stability index $\beta \in (0,1).$ The natural observation process associated with the state process \eqref{Conjecture1_state} with invented time-change  has the form
\begin{equation}
dV_t=h(X_t)dT_t+dW_{T_t}, ~ V_0=0. \label{Conjecture1_observ}
\end{equation}

In the theorem below we assume that the input data of this filtering model satisfy the following conditions:
\begin{enumerate}
\item[(C1)] the vector-functions $b(x)$ and $h(x)$ and $n\times m$-matrix-function $\sigma(x)$ are infinite differentiable and bounded;
\item[(C2)] the time-change process $T_t$ and Brownian motions $B_t$ and $W_t$  are independent processes;
\item[(C3)] the initial random vector $X_0$ is independent of processes $B_t,$ $W_t,$ and $T_t$ and has an infinite differentiable density function $p_0(x)$ decaying at infinity faster than any power of $|x|.$
\end{enumerate}
We note that the conditions on infinite differentiability of $b(x), \, h(x),$ and $\sigma(x)$ in (C1), as well as of the density function $p_0(x)$ in (C3) and its decay condition near infinity can be weakened.

The filtering problem \eqref{Conjecture1_state}, \eqref{Conjecture1_observ} is closely related to the filtering problem whose state  process is given by the following (non time-changed ) \^Ito stochastic differential equation
 \begin{equation}
\label{filtering}
dY_t=b(Y_t)dt+\sigma(Y_t)dB_{t}, ~ Y_{t=0}=X_0, \, t>0,
\end{equation}
and whose corresponding observation process is given by
\begin{equation}
dZ_t=h(Y_t)dt+dW_{t}, ~ Z_0=0. \label{observ}
\end{equation}
Introduce the process
\[
\rho(t)=\exp\{-\sum_{k=1}^m \int_0^th_k(Y_s)dW_s-\frac{1}{2}\int_0^t |h(Y_s)|^2ds\}
\]
and the probability measure $dP_0=\rho(t)dP.$ Further, let
\begin{equation} \label{lambda}
{\Lambda}_{t}=\hat{E} \big( \frac{dP}{dP_0}  \big| {\mathcal{Z}}_t \big),
\end{equation}
where the expectation $\hat{E}$ is under the reference measure $P_0.$ Then, as is known, the optimal filtering solution of the filtering problem  \eqref{filtering}, \eqref{observ} is given by the following Kallianpur-Striebel's formula (see, e.g. \cite{Popa,Rozovskii})
\[
E[f(Y_t)|{\mathcal{Z}}_t]=\frac{\hat{E}[f(X_t)\Lambda_t|{\mathcal{Z}}_t]}{\hat{E}[\Lambda_t|{\mathcal{Z}}_t]}.
\]
It is also known (see \cite{Popa,Rozovskii}) that under the reference measure $P_0,$ the process $Z_t$ is a standard Brownian motion independent of $Y_t,$ and $\Lambda_t$ has the representation
\[
\Lambda_t=1+\sum_{k=1}^m\int_0^t \Lambda_sh_k(Y_s) dZ^k_s.
\] 
Moreover, under conditions (C1)-(C3) 
the unnormalized filtering measure $p_t(f)=\hat{E}[f(Y_t)\Lambda_t|{\mathcal{Z}}_t]$ satisfies the following stochastic differential equation called the Zakai equation
\begin{equation} \label{Zakai1}
p_t(f)=p_0(f)+\int_0^t p_s(Af)ds+\sum_{k=1}^m\int_0^tp_s(h_kf)dZ^{k}_s,
\end{equation}
where $A$ is a second order elliptic differential operator given by equation \eqref{operator}.

Further, introducing the filtering density $U(t,x)$ through 
\begin{equation}
p_t(f)=\int_{R^n}f(x)U(t,x) dx
\end{equation}
one can show that $U(t,x)$ solves the (adjoint) Zakai equation

\begin{equation}\label{zakai}
dU(t,x)=A^*U(t,x) dt + \sum_{k=1}^m h_k(x)U(t,x) dZ_k(t),
\end{equation}
with the initial condition $U(0,x)=p_0(x).$ Here $A^{*}$ is the dual operator of $A.$

\begin{theorem} 
\label{Theorem_1}
Let $T_t$ be the inverse to a stable L\'evy subordinator $D_t$ and $\phi_t(f)=\hat{E}[f(X_t)\Lambda_{T_t}|{\mathcal{V}}_t],$ where $\mathcal{V}$ is the filtration generated by $V_t.$ Suppose 
conditions (C1)-C(3) are verified.
Then $\phi_t(f)$ satisfies the following Zakai type equation corresponding to filtering problem \eqref{Conjecture1_state}-\eqref{Conjecture1_observ}: 
\begin{equation} \label{FrZakai10}
\phi_t(f)=p_0(f)+J^{\beta} \phi_t(Af) + 
 \sum_{k=1}^m  \int_0^t  \phi_s(h_k f) dZ^k_{T_s}
 \end{equation}
Moreover, equation \eqref{FrZakai10} has a unique solution represented in the form
\begin{equation}\label{frzaksol}
\phi_t(f)=\int_0^{\infty} g_t(\tau) p_{\tau}(f) d\tau,
\end{equation}
where $g_t(\tau)$ is the density function of the process $T_t$ and $p_t(f)$ is the unnormalized filtering distribution of the Zakai equation \eqref{Zakai1} corresponding to the filtering model \eqref{filtering}-\eqref{observ}.
\end{theorem}

\begin{proof}
Let conditions (C1)-(C3) be verified. Then, in particular, the conditions for the existence of an unnormalized filtering distribution $p_t(f)=\hat{E}[f(Y_t)\Lambda_t|{\mathcal{Z}}_t]$ which  solves the Zakai equation \eqref{Zakai1}, is also verified. Here $Y_t$ is a solution to stochastic differential equation \eqref{filtering}. According to Theorem 3.3 in \cite{HKU} the time-changed process $X_t=Y_{T_{t}}$ solves stochastic differential equation \eqref{Conjecture1_state}. 

The connection $X_t=Y_{T_t}$ between the state processes $X_t$ and $Y_t$ implies the connection $V_t=Z_{T_t}$ between the observation processes $V_t$ and $Z_t.$ Indeed, letting $T_t=\tau,$ or the same $D_{\tau}=t,$ one obtains from the relation $dV_t=h(Y_{T_t})dT_t+dW_{T_t}$ and \eqref{observ} that $Z_{\tau}=V_{D_{\tau}},$ or the same $V_t=Z_{T_t}.$ It follows that the filtration $\mathcal{V}_t$  coincides with the filtration  $\mathcal{Z} \circ \mathcal{T}_t \equiv \mathcal{Z}_{\mathcal{T}_t}$  generated by the time changed observation process $Z_{T_t}.$

Further, the unnormalized filtering distribution $\phi_t(f)=\hat{E}[f(X_t)\Lambda_{T_t}|{\mathcal{Z\circ T}}_t]$  corresponding to the filtering problem \eqref{Conjecture1_state} - \eqref{Conjecture1_observ} is determined by conditioning on $T_t,$ namely $\phi_t(f)=E\Big[ \hat{E}[f(Y_\tau)\Lambda_{\tau}|{\mathcal{Z}}_\tau]\Big| \tau=T_t \Big].$ The latter can be written in the form
\begin{equation}
\label{fractionalUFD}
\phi_t(f)=\int_0^{\infty} g_t(\tau)p_{\tau}(f)d\tau,
\end{equation}
where $g_t(\tau), ~ \tau \in (0,\infty),$ is the density function of the time-change process $T_t$ for each $t\in (0, \infty),$ defined in equation \eqref{relation2}.  
Applying the fractional integration operator $J^{\beta}$ to equation  \eqref{relation399}, we have
\[
g_t(\tau)-\lim_{t \to 0+}g_t(\tau)=-\frac{\partial}{\partial \tau}J^{\beta}_t g_t(\tau)-\frac{\delta_0(\tau)}{\Gamma(1-\beta)}J^{\beta}_t t^{-\beta},
\]
in the sense of distributions. Due to part (a) of Lemma \ref{lemproperties} we have $\lim_{t \to 0+}g_t(\tau)=\delta_0(\tau).$ This fact together with the equation $J^{\beta} t^{-\beta}=\Gamma(1-\beta)$ implies  
\begin{equation} \label{relation100}
g_t(\tau)=-\frac{\partial}{\partial \tau}J^{\beta}_t g_t(\tau).
\end{equation}
Now multiplying both sides of  equation \eqref{Zakai1} by $p_{\tau}(f)$ and integrating over the interval $(0, \infty)$ with respect to the variable $\tau,$ and taking into account the conditioning on the time change process $T_t$,  one obtains
\begin{equation} \label{relation101}
\phi_t(f)-p_0(f)=\int_0^{\infty} g_t(\tau) \Big[\int_0^{\tau} p_s(Af)ds \Big] d\tau + \sum_{k=1}^m\int_0^{T_t} p_s(h_kf)dZ^{(k)}_s.
\end{equation}
Due to relation \eqref{relation100} the first term on the right side of \eqref{relation101}  can be written in the form
\begin{align} 
\int_0^{\infty} g_t(\tau) \Big[\int_0^{\tau} p_s(Af)ds \Big] d\tau&= - \int_0^{\infty} \frac{\partial}{\partial \tau}J^{\beta}_t g_t(\tau) \Big[\int_0^{\tau} p_s(Af)ds \Big] d\tau \notag\\
&= - \int_0^{\infty}  \Big[\int_0^{\tau} p_s(Af)ds \Big] d  \Big[J^{\beta}_t g_t(\tau)\Big] \label{relation105}
\end{align}
Now the integration by parts in \eqref{relation105} implies
\begin{align} 
\int_0^{\infty} g_t(\tau) \Big[\int_0^{\tau} p_s(Af)ds \Big] d\tau&= 
J^{\beta}_t \int_0^{\infty} g_t(\tau)p_{\tau}(Af) d \tau 
= J_t^{\beta} \phi_t(Af), \label{relation106}
\end{align}
since $J^{\beta}_t g_t(\tau) \to 0$ as $\tau \to \infty,$ and $\int_0^{\infty} p_s(Af)ds$ is bounded due to conditions (C1)-(C3). 
The second term on the right side of \eqref{relation101} 
\begin{equation} \label{relation107}
 \sum_{k=1}^m\int_0^{T_t} p_s(h_kf)dZ^{(k)}_s =  \sum_{k=1}^m\int_0^{T_t} \hat{E}[h_k (Y_s) f(Y_s)\Lambda_t|{\mathcal{Z}}_s] dZ^{(k)}_s
 \end{equation}
Using the change of variable formula (see \cite{Jacod}, Proposition 10.21) 
\[
\int_0^{T_t} H_s dZ_s=\int_0^t H_{T_{s-}}dZ_{T_s},
\]
for stochastic integrals driven by a semimartingale, we obtain\footnote{there are no jumps in this particular case}
\begin{align} 
 \sum_{k=1}^m\int_0^{T_t} p_s(h_kf)dZ^{(k)}_s 
 & =  \sum_{k=1}^m\int_0^{t} \hat{E}[h_k (Y_{T_s}) f(Y_{T_s})\Lambda_{T_s}|{\mathcal{Z}}_{T_s}] dZ^{(k)}_{T_s} \notag \\
 &= \sum_{k=1}^m\int_0^{t} \hat{E}[h_k (X_{s}) f(X_{s})\Lambda_{T_s}|{\mathcal{Z}}_{T_s}] dZ^{(k)}_{T_s} \notag \\ 
 & = \sum_{k=1}^m\int_0^{t} \phi_s (h_k f) dZ^{(k)}_{T_s}. \label{relation110}
 \end{align}
Equations \eqref{relation100}, \eqref{relation106}, and \eqref{relation110} imply the desired equation \eqref{FrZakai10}. 

\end{proof}

\begin{remark}
\begin{itemize}
\item[(1)] The process $Z_{T_t}$ is a semimartingale, therefore SDE \eqref{FrZakai10} is meaningful (see, e.g. \cite{Protter}).
\item[(2)] The time-changed processes $B_{T}$ and $W_T$ are not Markovian and has no independent increments. Therefore, the model \eqref{Conjecture1_state}, \eqref{Conjecture1_observ} can be applied to a class of correlated filtering processes. We note also that the classical Zakai equation is recovered when $\beta \to 1.$ 
\end{itemize}
\end{remark}

The differential form of \eqref{FrZakai10} involves a factional derivative in the Riemann-Liuoville sense. 
The differential form of the Zakai equation \eqref{FrZakai10} is
\begin{equation} \label{Zakai200}
d\phi_t(f)= {\mathcal{D}}_t^{1-\beta} \phi_t(Af) dt+ \sum_{k=1}^m  \phi_t (h_k f) dZ^{(k)}_{T_t}, ~~ \phi_{t=0}(f)=p_0(f).
\end{equation}
Stochastic partial differential equations  of the form  \eqref{Zakai200} (or  \eqref{FrZakai10} in the integral form), which involve fractional order derivatives (or fractional order integrals in its integral form)  will be called \textit{fractional Zakai type equations.}

Further, we introduce a filtering density and derive the adjoint Zakai equation generalizing the equation \eqref{zakai}. Let $\Phi(t,x)$ be defined as a generalized function 
\begin{equation} \label{Fr_density}
\phi_t(f)=\int_{R^n}f(x)\Phi(t,x)dx,
\end{equation}
for arbitrary infinite differentiable function $f$ with compact support. The function $\Phi(t,x)$ is called a \textit{filtering density associated with the filtering measure $\phi_t(f).$}

\begin{theorem} \label{Zakai_1}
Let the conditions (C1)-(C3) are verified. Then the filtering density $\Phi(t,x)$ associated with the filtering measure $\phi_t(f)$ in equation \eqref{Zakai200} satisfies the following fractional Zakai type equation
\begin{equation} \label{Zakai210}
d\Phi(t,x)= {\mathcal{D}}_t^{1-\beta} A^{\ast} \Phi(t,x) dt+ \sum_{k=1}^m  h_k(x)\Phi(t,x) dZ^{(k)}_{T_t}, 
\end{equation}
with the initial condition
$\Phi(0,x)=P(X_0=x).$
\end{theorem}

The proof of this theorem immediately follows from equation \eqref{Zakai200} substituting $\phi_t(f)$
in \eqref{Fr_density}.

Now we generalize the results of Theorem \ref{Theorem_1} and \ref{Zakai_1} to two different cases of filtering problems whose either state or observation processes are driven by a time-changed L\'evy process. Without  time-changed driving processes these two cases were discussed in Section \ref{Sec: Generalization}. Let $L_t$ be the L\'evy process given by equation \eqref{levy-ito}. The first case is the fractional filtering problem the state process of which is given by stochastic differential equation \eqref{stateLevyFrac} and the observation process of which is given by equation \eqref{Conjecture1_observ}.
Suppose that the input data of this filtering model satisfy the conditions:
\begin{enumerate}
\item[$(C1)^{'}$] the vector-functions $b(x), \, h(x), \, H(x,w), \, K(x,w)$ and  the matrix-function $\sigma(x)$ are infinite differentiable and bounded;
\item[$(C2)^{'}$] the time-change process $T_t,$ Brownian motions $B_t$ and $W_t,$ and Poisson random measures $\tilde{N}(t,\cdot)$ and $N(t,\cdot)$  are independent processes;
\item[$(C3)^{'}$] the initial random vector $X_0$ is independent of processes $B_t,$ $W_t,$ $\tilde{N}(t,\cdot),$ $N(t,\cdot)$$T_t$ and has an infinite differentiable density function $p_0(x)$ decaying at infinity faster that any power of $|x|.$
\end{enumerate}

\begin{theorem} 
\label{Theorem_2}
Let $T_t$ be the inverse to a stable L\'evy subordinator $D_t$ and $\phi_t(f)=\hat{E}[f(X_t)\Lambda_{T_t}|{\mathcal{Z\circ T}}_t],$ where $\mathcal{T}$ is the filtration generated by $T_t$ and $\Lambda_t$ is defined by \eqref{lambda}. Suppose 
conditions $(C1)^{'}-(C3)^{'}$ are verified.
Then $\phi_t(f)$ satisfies the following Zakai type equation corresponding to filtering problem \eqref{stateLevyFrac}, \eqref{Conjecture1_observ}: 
\begin{equation} \label{FrZakai300}
\phi_t(f)=p_0(f)+J^{\beta} \phi_t(\mathcal{A} f) + 
 \sum_{k=1}^m  \int_0^t  \phi_s(h_k f) dZ^k_{T_s},
 \end{equation}
 where the operator $\mathcal{A}$ is a pseudo-differential operator with the symbol $\Psi(x,\xi)$ given by \eqref{levysymbol}. 
Moreover, equation \eqref{FrZakai300} has a unique solution represented in the form
\begin{equation}\label{frzaksol}
\phi_t(f)=\int_0^{\infty} g_t(\tau) p_{\tau}(f) d\tau,
\end{equation}
where $g_t(\tau)$ is the density function of the process $T_t$ and $p_t(f)$ is the unnormalized filtering measure of the Zakai equation \eqref{Zakai11} corresponding to the filtering model \eqref{stateLevy}, \eqref{obser_1case}.
\end{theorem}

The second case is the filtering problem 
whose state process is driven by a time-changed Brownian motion and the observation process is driven by a time-changed L\'evy process. Namely, let the state process be given by equation \eqref{Conjecture1_state} and the observation process be given by the following stochastic differential equation
\begin{equation} \label{Observ400}
dH_t=h(X_t)dT_t+dW_{T_t}+\int_{R} wN_{\lambda}(dT_t,dw), ~ V_0=0,
\end{equation}
where the random measure $N_{\lambda}$ has the predictable compensator $\lambda(t,X_t,w)dt\nu(dw)$ with a L\'evy measure $\nu.$ Introduce the process 
\begin{align} 
\hspace{-0.7cm} 
{\cal{L}}_t &= \exp\{-\sum_{k=1}^m \int_0^th_k(Y_s)dW_s \notag \\
                  &- \frac{1}{2}\int_0^t |h(Y_s)|^2ds \notag 
      + \int_0^t\int_{R\setminus \{0\}}\ln \lambda(s,X_s, w ) N_{\lambda}(ds,dw) \\ 
                  &+ \int_0^t \int_{R \setminus \{0\}}(1-\lambda(s,X_s, w))ds \nu(d w)\}.   \label{relation400}
\end{align}

\begin{theorem} 
\label{Theorem_3}
Let $T_t$ be the inverse to a stable L\'evy subordinator $D_t$ and $\phi_t(f)=\hat{E}[f(X_t){\cal{L}}_{T_t}|{\mathcal{H}}_t],$ where $\mathcal{H}$ is the filtration generated by the process $H_t$ in equation \eqref{Observ400} and ${\cal{L}}_t$ is defined by \eqref{relation400}. Suppose 
conditions $(C1), (C2)^{'}$ and $(C3)^{'}$ are verified.
Then $\phi_t(f)$ satisfies the following Zakai type equation corresponding to filtering problem \eqref{Conjecture1_state}, \eqref{Observ400}: 
\begin{align} 
\phi_t(f) 
              &=p_0(f)+J^{\beta} \phi_t(Af) + 
 \sum_{k=1}^m  \int_0^t  \phi_s(h_k f) dZ^k_{T_s} \notag \\ 
              &+ \int_0^t \int_{R \setminus \{0\}}p_s\Big((\lambda(s, \cdot, w)-1)f\Big)d\hat{N}_T(ds,dw),\label{FrZakai320}
 \end{align}
 where $\hat{N}_T(ds,dw)=N_T(ds,dw)-dT_s \nu(dw)$ and the operator $A$ is defined in equation \eqref{operator}.
 Moreover, equation \eqref{FrZakai320} has a unique solution represented in the form
\begin{equation}\label{frzaksol}
\phi_t(f)=\int_0^{\infty} g_t(\tau) p_{\tau}(f) d\tau,
\end{equation}
where $g_t(\tau)$ is the density function of the process $T_t$ and $p_t(f)$ is the unnormalized filtering measure of the Zakai equation adjoint to \eqref{Zakai2} corresponding to the filtering model \eqref{state_2case}, \eqref{observ_2case}.
\end{theorem}

Proofs of Theorems \ref{Theorem_2} and \ref{Theorem_3} are similar to the proof of Theorem \ref{Theorem_1}. 



\vspace{3mm}

\appendix


\end{document}